\documentclass{amsart}

\usepackage{amsmath}
\usepackage{amsthm}
\usepackage{amsfonts}
\usepackage{amssymb}
\usepackage{enumerate}

\newcommand{\indicator}[1]{\ensuremath{\mathbf{1}_{\{#1\}}}}
\newcommand{\oindicator}[1]{\ensuremath{\mathbf{1}_{#1}}}

\newcommand{\E}{\mathbb{E}}
\newcommand{\Prob}{\mathbb{P}}
\newcommand{\C}{\mathbb{C}}
\newcommand{\R}{\mathbb{R}}

\newcommand{\T}{\mathrm{T}}  

\newcommand{\Tr}{\mathrm{Tr}}
\newcommand{\tr}{\mathrm{tr}}

\renewcommand\Re{\operatorname{Re}}
\renewcommand\Im{\operatorname{Im}}

\theoremstyle{plain}
  \newtheorem{theorem}{Theorem}

  \newtheorem{proposition}[theorem]{Proposition}
  
  \newtheorem{lemma}[theorem]{Lemma}

\theoremstyle{definition}
  \newtheorem{definition}[theorem]{Definition}
  
  \newtheorem{remark}[theorem]{Remark}

\begin{document}
\title[Fluctuations of Matrix Entries]{Fluctuations of Matrix Entries of Analytic Functions of Non-Hermitian Random Matrices}

\author[S. O'Rourke]{Sean O'Rourke}
\address{Department of Mathematics, Rutgers, Piscataway, NJ 08854  }
\email{sdo21@math.rutgers.edu}

\begin{abstract}
Consider an $n \times n$ non-Hermitian random matrix $M_n$ whose entries are independent real random variables.  Under suitable conditions on the entries, we study the fluctuations of the entries of $f(M_n)$ as $n$ tends to infinity, where $f$ is analytic on an appropriate domain.  This extends the results in \cite{ORS}, \cite{ORS2}, and \cite{PRS} from symmetric random matrices to the non-Hermitian case.  
\end{abstract}

\maketitle

\section{Introduction and Main Results}

There have been a number of results recently concerning matrix entries of functions of random matrices.  The results so far have dealt with symmetric (Hermitian) random matrix ensembles. That is, for a $n \times n$ real symmetric (Hermitian) random matrix, $M_n$, one studies the entries of the matrix $f(M_n)$ where $f$ is a regular test function.  

For a symmetric (Hermitian) matrix $M_n$, we define $f(M_n)$ using functional calculus.  That is, if $M_n$ is diagonalized as $M_n = V^\ast D V$ where $V$ is an orthogonal (unitary) matrix and $D$ is diagonal, then $f(M_n)$ is constructed as $f(M_n) = V^\ast f(D) V$ where $f(D)$ is again a diagonal matrix with the entries $f(D_{ii})$ on the diagonal.  

In \cite{LP}, Lytova and Pastur consider the case where $M_n$ is drawn from the Gaussian Orthogonal Ensemble (GOE) or Gaussian Unitary Ensemble (GUE).  The more general case where $M_n$ is a symmetric (Hermitian) Wigner matrix with i.i.d. entries (not necessarily Gaussian) was studied by Pizzo, Renfrew, and Soshnikov in \cite{PRS} and by Pastur and Lytova in \cite{LP2}.  

The results of \cite{PRS} and \cite{LP2} are extended in \cite{ORS} to the case where $M_n$ is a Wigner matrix with independent but not necessarily identically distributed entries.  The authors require that the off-diagonal entries have uniformly bounded fourth moments, diagonal entries have uniformly bounded second moments, and certain Lindeberg type conditions for both the off-diagonal entries and the diagonal entries are satisfied.  The test function $f(x)$ is assumed to satisfy
\begin{equation*}
	\int_\R (1 + 2|k|)^{2\*s} |\hat{f}(k)|^2 dk < \infty,
\end{equation*}
for some $s>3$ where $\hat{f}$ is the Fourier transform
\begin{equation*}
	\hat{f}(k) = \frac{1}{\sqrt{2 \pi}} \int_\R e^{-ikx} f(x) dx.
\end{equation*}
Under these conditions it is shown in \cite{ORS} that
\begin{equation} \label{eq:intro_idea}
 \sqrt{n} \left( f \left(\frac{M_n}{\sqrt{n}} \right)_{ij} - \E f \left(\frac{M_n}{\sqrt{n}} \right)_{ij} \right) - \alpha(f) (M_n)_{ij} \longrightarrow N(0, d^2(f)),
\end{equation}
as $n \rightarrow \infty$, where $\alpha(f)$ and $d^2(f)$ are given constants depending on the function $f$.  

Recently, similar results were obtained in \cite{ORS2} for the sample covariance case where $M_n = A_n A_N^\ast $ and $A_n$ is an $n \times N$ random matrix with independent entries under analogous moment and Lindeberg type conditions on the entries of $A_n$.

Instead of considering a specific entry of $f(M_n)$, Bai and Pan in \cite{BP} study the fluctuations of $u^\ast f(M_n) u$ when $M_n$ is a Wigner matrix whose entries satisfy some moment assumptions, $f$ is an analytic function, and $u$ is a delocalized unit vector.  

In this paper, we consider non-Hermitian random matrices with independent entries.  Given a $n \times n$ random matrix $M_n$, we define the empirical spectral distribution (ESD), $F^{M_n}$, of $M_n$ to be
$$ F^{M_n}(x,y) := \frac{1}{n} \#\{1 \leq j \leq n : \Re(\lambda_j) \leq x, \Im(\lambda_j) \leq y \}, $$
where $\lambda_1, \ldots, \lambda_n$ are the eigenvalues of $M_n$.  

The ESD for non-Hermitian random matrices with i.i.d. Gaussian entries was first studied by Mehta \cite{M}.  In particular, in the case where the entries of $M_n$ are i.i.d. standard complex normal random variables, the ESD of $\frac{1}{\sqrt{n}}M_n$ converges, as $n \rightarrow \infty$, to the circular law $F_{c}$ given by
$$F_{c}(x,y):= \frac{1}{\pi} \operatorname{mes}\Big(|z|\le 1: \Re(z)\le x, \Im(z)\le y\Big).$$
In other words, $F_c$ is the uniform distribution over the unit disk in the complex plane.  

Mehta used the joint density function of the eigenvalues of $\frac{1}{\sqrt{n}}M_n$ which was derived by Ginibre \cite{Gi}.  The real Gaussian case was studied by Edelman in \cite{Ecir}.  For the general (non-Gaussian) case when there is no formula, the problem appears much more difficult. Important results were obtained by Girko \cite{G1,G2}, Bai \cite{Bcir,BS}, and more recently by G\"otze and Tikhomirov \cite{GT}, Pan and Zhou \cite{PZ}, and Tao and Vu \cite{TVcir}. These results confirm the same limiting law under some moment or smoothness assumptions on the distribution of the entries. Recently, Tao and Vu (appendix by Krishnapur) were able to remove all these additional assumptions, establishing the law under the first two moments \cite{TVuniv}.

\begin{theorem}[Circular law for non-Hermitian i.i.d. matrices; \cite{TVuniv}] \label{thm:cir}
Assume that the entries of the $n \times n$ matrix $M_n$ are i.i.d. copies of a complex random variable of mean zero and variance one. Then the ESD of the matrix $\frac{1}{\sqrt{n}}M_n$ converges almost surely to $F_{c}$ as $n \rightarrow \infty$.
\end{theorem}
Theorem \ref{thm:cir} also holds for random matrices with independent but not necessarily identically distributed entries, provided one has some uniform control of the entries (see Appendix C of \cite{TVuniv}).  

When $M_n$ is a non-Hermitian matrix and $f$ is analytic, we define $f(M_n)$ as
\begin{equation*}
	f(M_n) = \sum_{k=0}^\infty f^{(k)}(0) \frac{M_n^k}{k!}
\end{equation*}
provided 
\begin{equation*}
	\sum_{k=0}^\infty \left| f^{(k)}(0) \right| \frac{\|M_n^k\|}{k!} < \infty,
\end{equation*}
where $\| \cdot \|$ denotes the spectral norm.  

In \cite{RS}, Rider and Silverstein study linear eigenvalue statistics of non-Hermitian random matrices with i.i.d. entries.  Let $f$ be analytic in a neighborhood of the disk $\{z \in \C : |z| \leq 4 \}$.  It is shown in \cite{RS} that 
\begin{equation} \label{eq:trfm}
	\tr\left[ f \left( \frac{1}{\sqrt{n}} M_n \right)\right] - n f(0)
\end{equation}
converges to a mean zero complex Gaussian as the size of the matrix tends to infinity, provided the distribution of the complex entries of $M_n$ satisfy some moment and smoothness assumptions.  The real case was studied in \cite{NP} under a different set of assumptions on the entries of $M_n$ in the case when $f$ is a polynomial.  

The goal of this paper is to prove the analogue of \eqref{eq:intro_idea} in the case that $M_n$ is a non-Hermitian random matrix with independent entries.  In particular, we consider the following ensemble of non-Hermitian random matrices.  

\begin{definition} \label{def:c0}
For each $n \geq 1$, let $M_n = (m_{nij})$ be a $n \times n$ real random matrix.  The sequence of random matrices $\{M_n\}_{n \geq 1}$ is said to satisfy condition {\bf C0} if the following conditions hold:
\begin{enumerate}[(i)]
\item For each $n \geq 1$, $\{ m_{nij} : 1 \leq i, j \leq n \}$ is a collection of independent real random variables with mean zero and variance one,  
\item $\sup_{n,i,j} \E |m_{nij}|^4 < \infty$,
\item For all $\epsilon>0$, 
$$\frac{1}{n^2} \sum_{i,j=1}^n \E |m_{nij}|^4 \indicator{|m_{nij}| > \epsilon\sqrt{n}} \longrightarrow 0$$
as $n \rightarrow \infty$. \label{lindeberg_condition}
\end{enumerate}
\end{definition}

\begin{remark}
For each $n \geq 1$, let $M_n$ be an $n \times n$ matrix whose entries are i.i.d. copies of a real random variable $\xi$.  Then the sequence $\{M_n\}_{n \geq 1}$ satisifies condition {\bf C0} if $\xi$ has mean zero, unit variance, and $\E[\xi^4] = m_4 < \infty$.  
\end{remark}

For each $n \geq 1$, let $M_n = (m_{nij})$ be a $n \times n$ real random matrix and assume the sequence of random matrices $\{M_n\}_{n \geq 1}$ satisfies condition {\bf C0}.  Let $f$ be analytic in a region that contains the disk $\{ z \in \C : |z| \leq 2+\epsilon \}$ for some $\epsilon>0$.  In our main result below, we show that for any fixed indices $i$ and $j$, the entry
\begin{equation} \label{eq:fmi}
	f \left( \frac{1}{\sqrt{n}}M_n \right)_{ij} - f(0) \delta_{ij} - f'(0) \frac{1}{\sqrt{n}} m_{nij} 
\end{equation}
converges in distribution to a complex Gaussian random variable as $n$ tends to infinity, where $\delta_{ij}$ denotes the Kronecker delta.  The appearance of $f(0)$ is not surprising considering \eqref{eq:trfm} and since $f(0)$ is the average of $f$ against the circular law.  That is, 
$$ \frac{1}{\pi} \int_{\{z:|z| \leq 1\}}f(z) d^2 z = f(0) $$
because $f$ is analytic.  The $f'(0)$ term in \eqref{eq:fmi} is analogous to the $\alpha(f)$ term which appears in \eqref{eq:intro_idea}.  

\begin{theorem} \label{thm:main}
For each $n \geq 1$, let $M_n$ be an $n \times n$ real random matrix and assume the sequence $\{M_n\}_{n \geq 1}$ satisfies condition {\bf C0}.  Fix $l \geq 1$ and assume that for each $1 \leq i \leq l$ and every $\delta > 0$, 
\begin{equation} \label{eq:lf14}
	\frac{1}{n} \sum_{j=1}^n \left\{ \E\left[ m_{nij}^4 \indicator{|m_{nij}|> \delta n^{1/4}} \right] + \E\left[ m_{nji}^4 \indicator{|m_{nji}| > \delta n^{1/4}} \right] \right\} \longrightarrow 0
\end{equation}
as $n \rightarrow \infty$.  Then the following holds:
\begin{enumerate}[(i)]
\item if $f$ is analytic in a region that contains the disk $\{z \in \C : |z| \leq 2 + \epsilon \}$ for some $\epsilon > 0$, the normalized matrix entries 
$$ \left\{ \sqrt{n} \left( f \left( \frac{1}{\sqrt{n}} M_n \right)_{ij} - f(0) \delta_{ij} - f'(0) \frac{1}{\sqrt{n}} m_{nij} \right) : 1 \leq i,j \leq l \right\} $$
are independent in the limit $n \rightarrow \infty$,
\item if $f_1, \ldots, f_m$ are analytic in a region that contains the disk $\{z \in \C : |z| \leq 2 + \epsilon \}$ for some $\epsilon > 0$ and $1 \leq i,j \leq l$ are fixed, the random $m$-vector 
$$ \sqrt{n} \left(  f_k \left( \frac{1}{\sqrt{n}} M_n \right)_{ij} - f_k(0) \delta_{ij} - f_k'(0) \frac{1}{\sqrt{n}} m_{nij} \right)_{k=1}^m $$
converges in distribution to the mean zero multivariate complex Gaussian $(Z(f_1), \ldots, Z(f_m))$ as $n \rightarrow \infty$ defined by
\begin{equation} \label{eq:ezfp}
	\E[Z(f_p) \overline{Z(f_q)}] = \sum_{r=2}^\infty \frac{f_p^{(r)}(0)}{r!} \overline{\frac{f_q^{(r)}(0)}{r!}}.
\end{equation}
\end{enumerate}
\end{theorem}

\begin{remark}
We draw the reader's attention to one special case of Theorem \ref{thm:main}.  Consider the monomial $f(z) = z^t$ for some integer $t>1$.  In this case, the limiting distribution $Z(f)$ is standard Gaussian.  Moreover, if $f(z) = z^t$, $g(z) = z^s$ for integers $s,t>1$ with $s \neq t$, then \eqref{eq:ezfp} implies that the limiting distributions $Z(f)$ and $Z(g)$ are independent.  
\end{remark}

Theorem \ref{thm:main} is a corollary of Theorem \ref{thm:resolvent} below in which we study the entries of the resolvent
$$ \left( z I_n - \frac{1}{\sqrt{n}} M_n \right)^{-1} = \left( z - \frac{1}{\sqrt{n}} M_n \right)^{-1}, $$
where $I_n$ is the identity matrix of order $n$.  The result can be extended from the resolvent case to arbitrary analytic functions by using Cauchy's integral formula.  

Below and throughout the paper, we use $\sqrt{-1}$ for the imaginary unit and reserve $i$ for an index.

\begin{theorem} \label{thm:resolvent}
For each $n \geq 1$, let $M_n$ be an $n \times n$ real random matrix and assume the sequence $\{M_n \}_{n \geq 1}$ satisfies condition {\bf C0}. Fix $l \geq 1$ and assume that for each $1 \leq i \leq l$ and all $\delta > 0$ \eqref{eq:lf14} holds as $n \rightarrow \infty$.  Define the $l \times l$ matrix
\begin{equation} \label{def:Y}
Y_n(z) := \sqrt{n} \left[ \left(z - \frac{1}{\sqrt{n}} M_n \right)^{-1}_{ij} - \frac{1}{z}\delta_{ij} - \frac{m_{nij}}{z^2 \sqrt{n}} \right]_{i,j=1}^l.
\end{equation}
Let $D \subset \C$ be any compact set outside the disk $\{z \in \C : |z| \leq 2 + \epsilon \}$ for some $\epsilon >0$.  Define 
$$ S_n(z) := \sum_{i,j=1}^l \left[ \alpha_{ij} \left( \frac{Y_n(z)_{ij} + \overline{Y_n(z)_{ij}}}{2} \right) + \beta_{ij} \left( \frac{Y_n(z)_{ij} - \overline{Y_n(z)_{ij}}}{2 \sqrt{-1}} \right) \right]$$
where $\alpha_{ij}, \beta_{ij} \in \C$ for $1 \leq i,j \leq l$.  Then $\{S_n\}_{n \geq 1}$ is tight in the space of continuous functions on $D$ and converges weakly to the complex Gaussian process
$$ S(z) := \sum_{i,j=1}^l \left[ \alpha_{ij} \left( \frac{Y(z)_{ij} + \overline{Y(z)_{ij}}}{2} \right) + \beta_{ij} \left( \frac{Y(z)_{ij} - \overline{Y(z)_{ij}}}{2 \sqrt{-1}} \right) \right]$$
where $Y(z)= \{Y(z)_{ij}\}_{i,j=1}^l$ is the matrix valued complex Gaussian process defined by
$$ \E[Y(z)_{ij} \overline{Y(w)}_{ij}] = \frac{1}{z^2 \overline{w}^2} \frac{1}{z \overline{w} - 1}, $$
and $\E[Y(z)_{ij}] = 0$.  

In addition, for any finite $s \geq 1$, the entries $Y_n(z_k)_{i_k,j_k}$, $1 \leq i_k, j_k \leq l$, $1 \leq k \leq s$, $z_1, \ldots, z_s \in D$ are independent provided $(i_1, j_1),(i_2, j_2), \ldots, (i_s, j_s)$ are distinct.  
\end{theorem}

The fact that Theorem \ref{thm:resolvent} is stated for $D$ outside the disk $\{z \in \C : |z| \leq 2 + \epsilon \}$ (and Theorem \ref{thm:main} is stated for functions analytic in a region containing $\{z \in \C : |z| \leq 2 + \epsilon \}$) requires some explanation.  For simplicity, assume for each $n \geq 1$, $M_n$ is a $n \times n$ real random matrix whose entries are i.i.d. copies of a random variable with mean zero, unit variance, and finite fourth moment.  Since the spectral radius of $\frac{1}{\sqrt{n}} M_n$ converges to $1$ in probability as $n$ tends to infinity \cite[Theorem 5.18]{BS}, it is natural to conjecture that Theorem \ref{thm:resolvent} should hold for $D$ outside the disk $\{z \in \C : |z| \leq 1 + \epsilon \}$ (and Theorem \ref{thm:main} should hold for $f$ analytic in a region containing $\{z \in \C : |z| \leq 1 + \epsilon \}$).  

In order to prove Theorem \ref{thm:resolvent}, we need to ensure that $zI_n - \frac{1}{\sqrt{n}} M_n$ is invertible and obtain control of the resolvent.  Since the spectral radius converges to $1$, it is natural to work on the event that the spectral radius is less than $1+\epsilon/2$.  However, in the proof of Theorem \ref{thm:resolvent}, we must also consider the matrix $M_{n,l}$, which is constructed from the matrix $M_n$ by removing the first $l$ columns and $l$ rows.  The problem that arises is that control of the spectral radius of $M_n$ does not imply that we have similar control of the spectral radius of $M_{n,l}$.  On the other hand, $\|M_{n,l}\| \leq \|M_n\|$.  Thus we must deal with the spectral norm instead of the spectral radius.  Since $\frac{1}{\sqrt{n}} \|M_n\|$ converges in probability to $2$ as $n \rightarrow \infty$ (see Lemma \ref{lemma:norm} below), we have uniform control of the resolvent if we take $|z| > 2$.  This statement is made precise in Proposition \ref{prop:contra}.  

The paper is organized as follows.  In Section \ref{sec:overview}, we state our preliminary tools and prove Theorem \ref{thm:main}.  The proof of Theorem \ref{thm:resolvent} is divided into two parts and presented in Sections \ref{sec:finite} and \ref{sec:tight}.

\section{Overview and Preliminary Tools} \label{sec:overview}

In this section, we introduce some standard notation and present our preliminary tools.  

Throughout this paper, $n$ is an asymptotic parameter going to infinity.  We use $o(1)$ to denote any quantity that is bounded in magnitude by an expression that converges to zero as $n$ tends to infinity.  We use $X=O(Y)$ to denote the estimate $|X| \leq CY$ where the implied constant $C$ is independent of $n$.  We use $C$ throughout the paper to denote some positive constant which does not depend on $n$ that may change from line to line.  

We note that by condition (\ref{lindeberg_condition}) from Definition \ref{def:c0}, it follows that there exists a sequence $\epsilon_n \rightarrow 0$ such that
\begin{equation*} 
	\frac{1}{\epsilon_n^4 n^2} \sum_{i,j} \E |m_{nij}|^4 \indicator{ |m_{nij}| > \epsilon_n \sqrt{n}} \longrightarrow 0
\end{equation*}
as $n \rightarrow \infty$.  

We then have the following standard truncation lemma which can be found in \cite{ORS2}.

\begin{lemma} \label{lemma:truncation}
Assume that for each $n\geq 1$, $M_n$ is an $n \times n$ real random matrix and the sequence $\{M_n\}_{n \geq 1}$ satisifies condition {\bf C0}.  Fix $l \geq 1$ and write $[l] = \{1, 2, \ldots, l\}$.  Assume that for each $i \in [l]$ and for every $\delta>0$, \eqref{eq:lf14} holds as $n \rightarrow \infty$.  Then there exists a sequence $\{\tilde{M}_n\}_{n \geq 1}$ where $\tilde{M}_n$ is a $n \times n$ matrix with independent entries and there exists a sequence $\epsilon_n$ which tends to zero as $n$ tends to infinity such that
\begin{enumerate}[(i)]
	\item the entries $(\tilde{M}_n)_{jk}$ have mean zero and variance one,
	\item $\sup_{j,k \notin [l]} |(\tilde{M}_n)_{jk}| \leq \epsilon_n \sqrt{n}$,
	\item $\sup_{1 \leq j \leq n, i \in [l]} |(\tilde{M}_n)_{ij}| \leq \epsilon_n n^{1/4}$,
	\item $\sup_{1 \leq j \leq n, i \in [l]} |(\tilde{M}_n)_{ji}| \leq \epsilon_n n^{1/4}$,
	\item $\sup_{n,j,k} \E |(\tilde{M}_n)_{jk}|^4 < \infty$,
	\item $\Prob(M_n \neq \tilde{M}_n) \longrightarrow 0$ as $n \rightarrow \infty$.
\end{enumerate}
\end{lemma}

Using \cite[Theorem 5.9]{BS} and Lemma \ref{lemma:truncation}, we obtain the following bound on the spectral norm of $M_n$.

\begin{lemma} \label{lemma:norm}
For each $n \geq 1$, let $M_n$ be an $n \times n$ real random matrix and assume the sequence $\{M_n \}_{n \geq 1}$ satisfies condition {\bf C0}. Then
\begin{equation*}
	\limsup_{n \rightarrow \infty} \left\| \frac{M_n}{\sqrt{n}} \right\| \leq 2
\end{equation*}
in probability.
\end{lemma}

Assume Theorem \ref{thm:resolvent} for the moment.  We are now ready prove Theorem \ref{thm:main}.

\begin{proof}[Proof of Theorem \ref{thm:main}]
Fix $\epsilon >0$ and let $\gamma$ be the circle of radius $2 + \epsilon$ centered around the origin.  Let $D$ be a compact set outside the disk $\{z \in \C : |z| \leq 2 + \epsilon/2 \}$ that contains $\gamma$.  Suppose $f$ is analytic in a disk about the origin that contains the contour $\gamma$.  By Lemma \ref{lemma:norm} and Cauchy's integral formula, we have, with probability going to one, that
$$ \sqrt{n} \left(  f \left( \frac{1}{\sqrt{n}} M_n \right)_{ij} - f(0) \delta_{ij} - f'(0) \frac{1}{\sqrt{n}} m_{nij} \right) = \frac{1}{2 \pi \sqrt{-1}} \int_\gamma f(z) Y_n(z)_{ij} dz $$
for $1 \leq i,j \leq l$.  To study the joint distribution of the $l^2$ entries, we use the Cramer-Wold device and study random variables of the form 
\begin{equation} \label{eq:cauchy}
	\frac{1}{2 \pi \sqrt{-1}} \int_\gamma \left( f_1(z) + \cdots + f_k(z) \right) S_n(z) dz,
\end{equation}
where $f_1, \ldots, f_k$ are analytic in a region about the origin that contains the contour $\gamma$.  By Thereom \ref{thm:resolvent}, $S_n(z)$ converges weakly to $S(z)$ in the space of continuous functions on $D$.  Therefore the random variable in \eqref{eq:cauchy} converges in distribution to a Gaussian random variable with mean zero as $n$ tends to infinity.  We now compute the limiting covariances.  Since $S(z)$ is a linear combination of the entries of $Y(z)$, it suffices to study the entries of $Y(z)$.  For functions $f$ and $g$ analytic in a region that contains the contour $\gamma$, the limiting covariance is given by 
$$ \E[Z(f)\overline{Z(g)}] = \E \left[ \frac{1}{2 \pi \sqrt{-1}} \int_\gamma f(z) Y(z)_{ij} dz \overline{ \frac{1}{2 \pi \sqrt{-1}} \int_\gamma f(w) Y(w)_{ij} dw } \right]. $$
Using Fubini's theorem, we can rewrite the limiting covariance as
\begin{align*}
	\E[Z(f)\overline{Z(g)}] &= \left( \frac{1}{2\pi} \right)^2 \int_\gamma \int_\gamma f(z) \overline{g(w)} \E[Y(z)_{ij}\overline{Y(w)_{ij}}]dz d\overline{w} \\
	&= \left( \frac{1}{2\pi} \right)^2 \int_\gamma \int_\gamma \frac{f(z)}{z^2} \overline{ \frac{g(w)}{w^2}} \frac{1}{z \overline{w} -1 } dz d\overline{w} \\
	&= \sum_{r=2}^\infty \frac{f^{(r)}(0)}{r!} \overline{ \frac{g^{(r)}(0)}{r!} }.
\end{align*}
Since $\E[Y(z)_{i_1j_1} Y(w)_{i_2,j_2}]=0$ for $(i_1,j_1) \neq (i_2,j_2)$ by Theorem \ref{thm:resolvent} and since the joint limiting distribution is Gaussian, it follows that the normalized entries 
$$ \left\{ \sqrt{n} \left( f \left( \frac{1}{\sqrt{n}} M_n \right)_{ij} - f(0) \delta_{ij} - f'(0) \frac{1}{\sqrt{n}} m_{nij} \right) : 1 \leq i,j \leq l \right\} $$
are independent in the limit $n \rightarrow \infty$.  The proof of Theorem \ref{thm:main} is complete.
\end{proof}

We now turn our attention to proving Theorem \ref{thm:resolvent}.  We will continue to use the notation $\| B \|$ to denote the spectral norm of the matrix $B$.  We will also make use of the following lemmas.  

\begin{lemma} \label{lemma:trace} Let $\mathbf{u}_n$ and $\mathbf{v}_n$ be independent $n$-vectors whose entries are independent random variables with mean zero and variance one.  Let $B_n$ be an independent $n \times n$ random matrix where $\|B\| \leq a$ for all $n$.  Then
\begin{equation*}
	\E \left| \frac{1}{\sqrt{n}}\mathbf{u}_n^\T B \mathbf{v}_n \right|^2 = \frac{1}{n} \E\Tr\left(BB^\T\right) \leq a^2.
\end{equation*}
\end{lemma}
\begin{proof}
The inequality is trivial and follows from the assumption $\|B\| \leq a$.  We denote the entries of $\mathbf{u}_n = (u_j)_{j = 1}^n$, $\mathbf{v}_n=(v_k)_{k=1}^n$, and $B = (B_{ij})_{i,j=1}^n$.  Then by the independence of $B$, $\mathbf{u}_n$, and $\mathbf{v}_n$, we have that
\begin{align*}
	\E \left| \frac{1}{\sqrt{n}}\mathbf{u}_n^\T B \mathbf{v}_n \right|^2 &= \frac{1}{n}\sum_{j,k,l,m=1}^n \E u_j B_{jk} v_k u_l B_{lm} v_m \\
	& = \frac{1}{n}\sum_{j,k} \E B_{jk}^2 = \frac{1}{n} \E \Tr(BB^\T).
\end{align*}
\end{proof}

\begin{lemma} \label{lemma:moments}
For each $n \geq 1$, let $\mathbf{x}_n=(x_{ni})_{i=1}^n$ be a real $n$-vector and let $B_n=(b_{nij})_{i,j=1}^n$ be an independent $n \times n$ random matrix.  Assume that
\begin{enumerate}[(i)]
\item The entries of $\mathbf{x}_n$ are independent and have mean zero and variance 1,
\item $\sup_{n,i} \E|x_{ni}|^4 \leq m_4 < \infty$,
\item There exists a sequence $\epsilon_n \rightarrow 0$ such that $|x_{ni}| \leq \epsilon_n \sqrt{n}$ for all $n \geq 1$ and $1 \leq i \leq n$,
\item There exists a constant $a$ (not depending on $n$) such that $\|B\| \leq a$. 
\end{enumerate}
Then there exists an absolute constant $C>0$ such that
\begin{enumerate}[(i)]
\item $\E|\frac{1}{n}\mathbf{x}_n^\T \mathbf{x}_n - 1 |^4 \leq C \left[\frac{m_4\epsilon_n}{n} + \frac{m_4^2}{n^2} \right]$, \label{item:lln} 
\item $\E|\frac{1}{n}\mathbf{x}_n^\T B_n \mathbf{x}_n - \frac{1}{n} \tr(B_n)|^4 \leq C a^4 m_4 \left( \frac{1}{n^2} + \frac{\epsilon_n}{n} \right)$, \label{item:bai}
\item For any fixed index $k$, $\E|\frac{1}{\sqrt{n}}e_k^\T B_n \mathbf{x}_n |^4 \leq C \left[ a^4 \frac{m_4+1}{n^2} \right]$, \label{item:lhs}
\item For any fixed index $k$, $\E|\frac{1}{\sqrt{n}}\mathbf{x}_n^\T B_n e_k |^4 \leq C \left[ a^4 \frac{m_4+1}{n^2} \right]$. \label{item:rhs}
\end{enumerate}
\end{lemma}
\begin{proof}
To prove \eqref{item:lln}, we note that 
\begin{align*}
	\E|\frac{1}{n}\mathbf{x}_n^\T \mathbf{x}_n - 1 |^4 = \frac{1}{n^4} \sum_{j_1, j_2, j_3, j_4 =1}^n\E[ (x_{nj_1}^2 - 1) (x_{nj_2}^2 - 1) (x_{nj_3}^2 - 1) (x_{nj_4}^2 - 1)].
\end{align*}
By the independence of the entries of $\mathbf{x}_n$, $\E[ (x_{nj_1}^2 - 1) (x_{nj_2}^2 - 1) (x_{nj_3}^2 - 1) (x_{nj_4}^2 - 1)]=0$ when any $j_l$ is distinct from the other three indices.  So we have two cases that give non-vanishing expectation:
\begin{enumerate}
\item When $j_1 = j_2 = j_3 = j_4$, we have that
\begin{equation*}
	 \frac{1}{n^4} \sum_{j=1}^n \E[ (x_{nj}^2 - 1)^4] \leq \frac{1}{n^4} n m_4 (\epsilon_n \sqrt{n})^4 \leq \frac{m_4}{n} \epsilon_n.
\end{equation*}

\item For $j_1 = j_l$ and $j_s = j_t$ with $j_s \neq j_1$, there are $3$ arrangements of such pairings.  For each pairing, we have that
\begin{equation*}
	\frac{1}{n^4} \sum_{j_1 \neq j_2} \E[(x_{nj_1}^2 - 1)^2 (x_{nj_2}^2 - 1)^2] \leq \frac{1}{n^4}n^2 m_4^2
\end{equation*}
\end{enumerate}

Thus, we have that
\begin{align*}
	 \E|\frac{1}{n}\mathbf{x}_n^\T \mathbf{x}_n - 1 |^4 \leq C \left[\frac{m_4\epsilon_n}{n} + \frac{m_4^2}{n^2} \right].
\end{align*}

Statement \eqref{item:bai} follows from Lemma B.26 in \cite{BS} (taking $p=4$) by conditioning on the entries of $B_n$ since $B_n$ and $\mathbf{x}_n$ are independent.  

To prove statement \eqref{item:lhs}, we write
\begin{align*}
	\E|\frac{1}{\sqrt{n}}e_k^\T B_n \mathbf{x}_n |^4 = \frac{1}{n^2} \sum_{j_1,j_2,j_3,j_4=1}^n \E[ B_{nkj_1} x_{nj_1} B_{nkj_2} x_{nj_2} \overline{B_{nkj_3}} x_{nj_3} \overline{B_{nkj_4}} x_{nj_4}].
\end{align*}
As in the proof of statement \eqref{item:lln}, we use the independence of the entries and consider two distinct cases when the expectation of the summand is non-zero.  The cases are: 
\begin{enumerate}
\item When $j_1 = j_2 = j_3 = j_4$, we have that
\begin{equation*}
	\frac{1}{n^2} \sum_{j=1}^n \E[ |B_{nkj}|^4 x_{nj}^4 ] \leq \frac{m_4}{n^2} a^2 (B_n B_n^\ast)_{kk} \leq \frac{m_4 a^4}{n^2}.
\end{equation*}
\item For $j_1 = j_l$ and $j_s = j_t$ with $j_s \neq j_1$, there are $3$ arrangements of such pairings.  In each pairing we obtain
\begin{equation*}
	\frac{1}{n^2} \sum_{j_1 \neq j_2} \E[ B_{nkj_1}^2 B_{nkj_2}^2 x_{nj_1}^2 x_{nj_2}^2] \leq \frac{1}{n^2} (B_n B_n^\ast)_{kk}^2 \leq \frac{a^4}{n^2}.
\end{equation*}
\end{enumerate}

Combining the two bounds above completes the proof of statement \eqref{item:lhs}.  The proof of statement \eqref{item:rhs} follows the same argument.  
\end{proof}

For the remainder of the paper, we fix $l \geq 1$, $\epsilon>0$, and a compact set $D \subset \C$ outside the disk $\{z \in \C : |z| \leq 2 + \epsilon \}$.  

We write the matrix $M_n$ as a block matrix in the following way: let $X$ be the upper-left $l \times l$ block of $M_n$, let $M_{n,l}$ be the lower-right $(n-l) \times (n-l)$ block of $M_n$, let $\phi$ be the lower-left $(n-l) \times l$ block of $M_n$, and let $\psi$ be the upper-right $l \times (n-l)$ block.  That is,
$$ M_n = \begin{bmatrix} X & \psi \\ \phi & M_{n,l} \end{bmatrix}. $$
Since the entries of $M_n$ are independent, $X, M_{n,l}, \phi,$ and $\psi$ are independent.  

We write the resolvent of $M_n$ as 
$$ R_n(z) := \left( z - \frac{1}{\sqrt{n}} M_n \right)^{-1}$$
and write the resolvent of $M_{n,l}$ as 
$$ R_{n,l}(z) := \left(z - \frac{1}{\sqrt{n}}M_{n,l} \right)^{-1}. $$ 

In order to work with $R_n(z)$ (or $R_{n,l}(z)$) we will need a priori control of the norm $\|R_n(z)\|$ ($\|R_{n,l}(z)\|$).  Following Rider and Silverstein in \cite{RS}, we define the event
\begin{equation*}
	\Omega_n = \left\{ M_n : \frac{1}{\sqrt{n}} \| M_n \| \leq \kappa \right\}
\end{equation*}
where we fix the value $2 < \kappa \leq 2+\epsilon/2$.  Then for $|z| \geq 2 + \epsilon$, there exists a constant $\mathcal{K} = \mathcal{K}(\kappa, \epsilon)$ such that
\begin{equation*}
	\sup_{|z|\geq 2+\epsilon} \|R_n(z)\| \leq \mathcal{K} \text{ on the event } \Omega_n.
\end{equation*}
Indeed, for $\frac{1}{\sqrt{n}} \|M_n \| \leq \kappa \leq 2 + \epsilon/2$, 
$$ \|R_n(z) \| = \left\| \frac{1}{z} \left( I_n - \frac{1}{z\sqrt{n}} M_n \right)^{-1} \right\| \leq \frac{1}{|z|} \sum_{k=0}^\infty \left \|\frac{M_n}{z\sqrt{n}} \right \|^k \leq \sum_{k=0}^\infty \left( \frac{2+\epsilon/2}{2 + \epsilon} \right)^k, $$
provided $|z| \geq 2+\epsilon$.  

We now assume both $\kappa$ and $\mathcal{K}$ are fixed and restate the above result in the following proposition.

\begin{proposition} \label{prop:contra}
If $\frac{1}{\sqrt{n}} \|M_n \| \leq \kappa$, then $\sup_{|z| \geq 2 + \epsilon} \|R_n(z)\| \leq \mathcal{K}$.  
\end{proposition}

By the contrapositive of Proposition \ref{prop:contra}, if $\sup_{z \in D}\|R_n(z) \| > \mathcal{K}$, then $\frac{1}{\sqrt{n}} \|M_n\| > \kappa$.  Thus, 
$$ \Prob \left(  \sup_{z \in D}\|R_n(z) \| > \mathcal{K} \right) \leq \Prob(\Omega_n^C) \longrightarrow 0 $$
as $n \rightarrow \infty$ by Lemma \ref{lemma:norm}.  

Therefore, when we work with the resolvent $R_n(z)$, we will always work on the event that $\sup_{z \in D}\|R_n(z) \|$ is finite.  For the remainder of the paper, we will assume $z \in D$.  

We also define the event 
\begin{equation*}
	\Omega_{n,l} = \left\{ M_n : \frac{1}{\sqrt{n}} \| M_{n,l} \| \leq \kappa \right\}.
\end{equation*}
In particular, if $\|M_n\| \leq \kappa$ then $\|M_{n,l} \| \leq \kappa$ and hence $\Omega_n \subset \Omega_{n,l}$.  Equivalently, we write $\oindicator{\Omega_n} \leq \oindicator{\Omega_{n,l}}$ where $\oindicator{\Omega}$ denotes the indicator function of the event $\Omega$.  

We divide the proof of Theorem \ref{thm:resolvent} into two parts.  In Section \ref{sec:finite}, we prove the finite dimensional distributions of $S_n(z)$ converge to $S(z)$.  We then show that $\{S_n\}_{n \geq 1}$ is tight in the space of continuous functions on $D$ in Section \ref{sec:tight}.

\section{Convergence of the Finite Dimensional Distributions} \label{sec:finite}

In this section, we show that the finite dimensional distributions of $S_n(z)$ converge to $S(z)$.  To study the entries of the resolvent we apply the standard formula for the inverse of a partitioned matrix (see for instance \cite{HJ}).  In particular, we obtain 
\begin{align*}
	\left[ \left(z - \frac{1}{\sqrt{n}}M_n \right)^{-1}_{ij} \right]_{i,j=1}^l \oindicator{\Omega_{n,l}} &= \left( z - \frac{1}{\sqrt{n}}{X} - \frac{1}{n} \psi R_{n,l}(z) \phi \right)^{-1} \oindicator{\Omega_{n,l}} \\
	&= \frac{1}{z} \left(I_l - \left[ \frac{1}{\sqrt{n}}X + \frac{1}{n} \psi R_{n,l}(z) \phi \right] \right)^{-1} \oindicator{\Omega_{n,l}}
\end{align*}
with probability going to one.  

It will follow from the discussion below and from Lemma \ref{lemma:trace} that 
$$ \|X\| + \|n^{-1/2}\psi R_{n,l}(z) \phi \|\oindicator{\Omega_{n,l}} $$
is bounded in probability.  Thus, we have that
$$ Y_n(z) \oindicator{\Omega_{n,l}} = \frac{\sqrt{n}}{z^2} \left[ \frac{1}{n} \psi R_{n,l}(z) \phi \right]\oindicator{\Omega_{n,l}} + o(1) $$
in probability.  

By Proposition \ref{prop:contra},
$$ \Prob \left( \|Y_n(z)\| \oindicator{\Omega_{n,l}^C} > \eta \right) \leq \Prob(\Omega_{n,l}^C) + \Prob \left(  \sup_{z \in D}\|R_n(z) \| > \mathcal{K} \right) \longrightarrow 0 $$
as $n \rightarrow \infty$.  Thus
$$ Y_n(z) = \frac{\sqrt{n}}{z^2} \left[ \frac{1}{n} \psi R_{n,l}(z) \phi \right]\oindicator{\Omega_{n,l}} + o(1) $$
in probability.  

We now apply Lemma \ref{lemma:truncation}.  Let $\tilde{M}_n$ be the truncated matrix in Lemma \ref{lemma:truncation} and define $\tilde{X}$, $\tilde{M}_{n,l}$, $\tilde{\phi}$, $\tilde{\psi}$ to be the appropriate blocks of $\tilde{M}_n$.  Since $M_n$ and $\tilde{M}_n$ agree with probability going to $1$, we have that
$$ Y_n(z) = \frac{\sqrt{n}}{z^2} \left[ \frac{1}{n} \tilde{\psi} R_{n,l}(z) \tilde{\phi} \right]\oindicator{\Omega_{n,l}} + o(1) $$
in probability.  

Our goal is to apply Lemma \ref{lemma:clt} from Appendix \ref{sec:clt} to the term $\tilde{\psi} R_{n,l}(z) \tilde{\phi}$.  However, we first need to show that 
$$ \frac{1}{n} \tr \left( R_{n,l}(z) R_{n,l}(z)^\mathrm{T} \right) \Omega_{n,l} $$
converges in probability to a constant as $n$ tends to infinity.  We do so in the following lemma.  

\begin{lemma} \label{lemma:variance}
For any $z,w \in D$, 
\begin{equation*}
	\frac{1}{n} \tr \left( R_n(z) R_n(w)^\T \right) \oindicator{\Omega_n} \longrightarrow \frac{1}{zw-1}
\end{equation*}
in probability as $n \rightarrow \infty$.  The result also holds when $R_n$ is replaced by $R_{n,l}$ and $\Omega_n$ is replaced by $\Omega_{n,l}$.  
\end{lemma}

\begin{proof}
We will prove the result for $R_n$ on the event $\Omega_n$.  The proof for $R_{n,l}$ follows the same argument. 

We define $X_n := \frac{1}{\sqrt{n}} M_n$ and recall that
$$ \Omega_n = \left\{ X_n : \|X_n \| \leq \kappa \right\}. $$
By Lemma \ref{lemma:truncation}, the matrix $X_n$ and the truncated version $\tilde{X}_n$ coincide with probability going to one.  Thus it suffices to prove Lemma \ref{lemma:variance} for the truncated version.  To avoid unnecessary notation, we let $X_n$ denote the truncated matrix and $R_n(z)$ denote the resolvent of the truncated matrix.  

Let $D(z) = (z-X_n)$ and $H(z,w) = [D(w)^\T D(z)]^{-1}$.  We will let $d_k(z)$ denote the $k$-th column of $D(z)$ and $D_k(z)$ denote the matrix $D(z)$ with the $k$-th column removed.  Then we have that
\begin{align*}
	\frac{1}{n} \tr H(z,w)\oindicator{\Omega_n} &= \frac{1}{n} \sum_{k=1}^{n-1} \frac{1}{d_k(w)^\T d_k(z) - d_k(w)^\T D_k(z) H_k(z,w)D_k(w)^\T d_k(z)}\oindicator{\Omega_n} \\
	& = \frac{1}{n} \sum_{k=1}^{n-1} \frac{1}{d_k(w)^\T d_k(z) - d_k(w)^\T D_k(z) H_k(z,w)D_k(w)^\T d_k(z) \oindicator{\Omega_n}}\oindicator{\Omega_n} \\
 	& = \frac{1}{n} \sum_{k=1}^{n-1} \beta_k.
\end{align*}
where $H_k(z,w) = [D_k(w)^\T D_k(z)]^{-1}$ and $\beta_k^{-1} = d_k(w)^\T d_k(z) - d_k(w)^\T D_k(z) H_k(z,w)D_k(w)^\T d_k(z)\oindicator{\Omega_n}$.  

Let $r_k$ denote the $k$-th column of $X_n$ and let $X_n^{(k)}$ denote the matrix $X_n$ with the $k$-th column removed.  On the event $\Omega_n$, we have that $\|X_n^{(k)}\| \leq \kappa$ and hence 
\begin{equation*}
	\|H(z,w)\| \leq \mathcal{K}^2 \qquad \text{and } \sup_{k} \|H_k(z,w)\| \leq \mathcal{K}^2.
\end{equation*}

We also note that
\begin{equation*}
	d_k(w)^\T d_k(z) = zw - (z+w){X_n}_{kk} + r_k^\T r_k.
\end{equation*}

Recalling that the entries of $X_n$ (and hence $r_k$) have variance $\frac{1}{n}$, we have that
\begin{align*}
	\Prob & \left( \sup_{1 \leq k \leq n-1} |d_k(w)^\T d_k(z) - (zw+1)| > \epsilon \right) \\
	& \qquad \leq \sum_{k=1}^{n-1} \left[\Prob\left( |{X_n}_{kk}| > \frac{\epsilon}{2|z|+2|w|} \right) + \Prob\left( |r_k^\T r_k - 1| > \frac{\epsilon}{2} \right) \right] \\
	& \qquad \leq \sum_{k=1}^{n-1} \left[ \frac{(2|z|+2|w|)^4 m_4}{n^2 \epsilon^4} + \frac{16}{\epsilon^4} \E|r_k^\T r_k - 1|^4 \right] = O(\epsilon_n)
\end{align*}
by Lemma \ref{lemma:moments}.  

Define the event
\begin{equation*}
	\Omega_{n}^{(k)} = \left\{ \|X_n \| : \|X_n^{(k)} \| \leq \kappa \right\}.
\end{equation*}
Below we will make use of the fact that $\oindicator{\Omega_n} \leq \oindicator{\Omega_{n}^{(k)}}$.

We note that
\begin{align*}
	& d_k(w)^\T D_k(z) H_k(z,w)D_k(w)^\T d_k(z) \\ 
	& \qquad = r_k^\T D_k(z) H_k(z,w)D_k(w)^\T r_k - we_k^\T D_k(z) H_k(z,w)D_k(w)^\T r_k \\
	& \qquad\qquad - z r_k^\T D_k(z) H_k(z,w)D_k(w)^\T e_k + zw e_k^\T D_k(z) H_k(z,w)D_k(w)^\T e_k.
\end{align*}

Since $\|D_k(z) H_k(z,w)D_k(w)^\T d_k(z)\|\oindicator{\Omega_n} \leq \kappa^2 \mathcal{K}^2$, by Lemma \ref{lemma:moments} and Markov's inequality, we have that
\begin{align*}
	\Prob &\left( \sup_{1 \leq k \leq n-1} | r_k^\T D_k(z) H_k(z,w)D_k(w)^\T r_k - \frac{1}{n} \Tr (D_k(z) H_k(z,w)D_k(w)^\T) |\oindicator{\Omega_n} > \epsilon \right) \\
	&  \leq \Prob \left( \sup_{1 \leq k \leq n-1} | r_k^\T D_k(z) H_k(z,w)D_k(w)^\T r_k - \frac{1}{n} \Tr (D_k(z) H_k(z,w)D_k(w)^\T) |\oindicator{\Omega_{n}^{(k)}} > \epsilon \right) \\ & = O(\epsilon_n), \\
	\Prob &\left( \sup_{1 \leq k \leq n-1} | e_k^\T D_k(z) H_k(z,w)D_k(w)^\T r_k |\oindicator{\Omega_n} > \epsilon \right) = O\left(\frac{1}{n} \right), \\
	\Prob &\left( \sup_{1 \leq k \leq n-1} | r_k^\T D_k(z) H_k(z,w)D_k(w)^\T e_k |\oindicator{\Omega_n} > \epsilon \right) = O\left(\frac{1}{n} \right).
\end{align*}
Furthermore, 
\begin{align*}
	\frac{1}{n} \Tr (D_k(z) H_k(z,w)D_k(w)^\T) = \frac{n-2}{n} = 1 + O\left(\frac{1}{n} \right).
\end{align*}

Therefore, the above estimates imply that
\begin{equation*}
	\sup_{1 \leq k \leq n-1} | \beta_k - zw + zw e_k^\T D_k(z) H_k(z,w)D_k(w)^\T e_k|\oindicator{\Omega_n} \longrightarrow 0
\end{equation*}
in probability as $n \rightarrow \infty$.  

Let $y_k$ denote the $k$-th row of $X_n$ with the $k$-th entry removed.  Then 
\begin{equation*}
	e_k^\T D_k(z) H_k(z,w)D_k(w)^\T e_k = y_k^\T H_k(z,w) y_k.
\end{equation*}
Let $\hat{H}_k(z,w) = [\hat{D}_k(w)^\T \hat{D}_k(z)]^{-1}$ where $\hat{D}_k(z)$ is the matrix $D_k(z)$ with the $k$-th row removed.  Since $H_k(z,w)^{-1} = \hat{H}_k(z,w)^{-1} + y_k y_k^\T$, by the resolvent identity
\begin{align*}
	y_k^\T H_k(z,w) y_k = 1 - \frac{1}{1 + y_k^\T \hat{H}_k(z,w) y_k }.
\end{align*}

Since for any $1 \leq k \leq n-1$, 
\begin{align*}
	\left|\frac{1}{\beta_k}\right| = |H(z,w)_{kk}|\oindicator{\Omega_n} \leq \mathcal{K}^2, \\
	\left| y_k^\T \hat{H}_k(z,w) y_k \right|\oindicator{\Omega_n}  \leq \kappa^2 \mathcal{K}^2,
\end{align*}
it follows that
\begin{align*}
	\sup_{1 \leq k \leq n-1} \left| \frac{1}{\beta_k} - \frac{1+ y_k^\T \hat{H}_k(z,w) y_k}{zw} \right|\oindicator{\Omega_n} \longrightarrow 0
\end{align*}
in probability as $n \rightarrow \infty$. 

Now define the event
\begin{equation*}
	\tilde{\Omega}_{n}^{(k)} = \left\{ X_n : \| \tilde{X}_n^{(k)} \| \leq \kappa \right\}
\end{equation*} 
where $\tilde{X}_n^{(k)}$ is the matrix $X_n$ with the $k$-th column and $k$-th row removed.  Again we will make use of the fact that $\oindicator{\Omega_n} \leq \oindicator{\tilde{\Omega}_{n}^{(k)}}$.  In particular, using Markov's inequality and Lemma \ref{lemma:moments}, it follows that
\begin{align*}
	\sup_{1 \leq k \leq n-1} &\left| y_k^\T \hat{H}_k(z,w) y_k - \frac{1}{n} \tr( \hat{H}_k(z,w)) \right|\oindicator{\Omega_n} \\ 
	& \leq \sup_{1 \leq k \leq n-1} \left| y_k^\T \hat{H}_k(z,w) y_k - \frac{1}{n} \tr( \hat{H}_k(z,w)) \right|\oindicator{\tilde{\Omega}_{n}^{(k)}} \longrightarrow 0
\end{align*}
in probability as $n \rightarrow \infty$.  

Finally, let $S_k(z,w) = [(w-\hat{X}_k)^\T(z-\hat{X}_k)]^{-1}$ where $\hat{X}_k$ is the matrix $X_n$ where the $k$-th row and $k$-th column are replaced by zeros.  In particular, $S_k(z,w)$ is an $n \times n$ matrix.  It follows that 
\begin{equation*}
	\frac{1}{n} \left| \tr S_k(z,w) - \tr \hat{H}_k(z,w)\right| = O\left( \frac{1}{n} \right).
\end{equation*}
Then by the resolvent identity, for $1 \leq k \leq n-1$, 
\begin{align*}
	& \left| \frac{1}{n} \tr S_k(z,w) - \frac{1}{n} H(z,w) \right|\oindicator{\Omega_n} \\
	& \qquad= \frac{1}{n} \tr\left[ H(z,w) \left( (w-X_n)^\T(z-X_n) - (w-\hat{X}_k)^\T (z-\hat{X}_k) \right) S_k(z,w) \right|\oindicator{\Omega_n} \\
	& \qquad \leq \frac{4}{n} \| H(z,w) \left( (w-X_n)^\T (z-X_n) - (w-\hat{X}_k)^\T (z-\hat{X}_k) \right) S_k(z,w) \| \oindicator{\Omega_n} \\
	& \qquad \leq \frac{8}{n} \mathcal{K}^4 (|z|+\kappa)(|w|+\kappa) 
\end{align*}
since $(w-X_n)^\T (z-X_n) - (w-\hat{X}_k)^\T (z-\hat{X}_k)$ is at most rank $4$.  Therefore, we have that
\begin{equation*}
	\sup_{1 \leq k \leq n-1} \left| \frac{1}{n} \tr \hat{H}_k(z,w) - \frac{1}{n} H(z,w) \right|\oindicator{\Omega_n} = O\left( \frac{1}{n} \right).
\end{equation*}

Therefore, combining the above, yields that
\begin{align*}
	\sup_{1 \leq k \leq n-1} \left| \frac{1}{\beta_k} - \frac{1 + \frac{1}{n} \tr H(z,w)}{zw} \right|\oindicator{\Omega_n} \longrightarrow 0
\end{align*}
in probability as $n \rightarrow \infty$ and hence
\begin{align*}
	\frac{1}{n}\tr H(z,w)\oindicator{\Omega_n} - \frac{1 + \frac{1}{n} \tr H(z,w)}{zw}\oindicator{\Omega_n} \longrightarrow 0
\end{align*}
in probability as $n \rightarrow \infty$.  The proof of the lemma is complete.  
	
\end{proof}

By Lemma \ref{lemma:variance}, we have that
\begin{align*}
	&\frac{1}{n} \tr \left( \frac{R_{n,l}(z)}{2z^2} + \frac{\overline{R_{n,l}(z)}}{2\overline{z}^2} \right) \left( \frac{R_{n,l}(w)}{2w^2} + \frac{\overline{R_{n,l}(w)}}{2\overline{w}^2} \right)^\T \oindicator{\Omega_{n,l}} \\
	& \qquad \longrightarrow \frac{1}{4} \left[ \frac{1}{z^2 w^2 (zw-1)} + \frac{1}{z^2 \overline{w}^2 (z\overline{w}-1)} + \frac{1}{\overline{z}^2 w^2 (\overline{z}w-1)} + \frac{1}{\overline{z}^2 \overline{w}^2 (\overline{z}\overline{w}-1)}  \right], \\
	&\frac{1}{n} \tr \left( \frac{R_{n,l}(z)}{2\sqrt{-1}z^2} - \frac{\overline{R_{n,l}(z)}}{2\sqrt{-1}\overline{z}^2} \right) \left( \frac{R_{n,l}(w)}{2\sqrt{-1}w^2} - \frac{\overline{R_{n,l}(w)}}{2\sqrt{-1}\overline{w}^2} \right)^\T \oindicator{\Omega_{n,l}} \\
	&\qquad \longrightarrow -\frac{1}{4} \left[  \frac{1}{z^2 w^2 (zw-1)} - \frac{1}{z^2 \overline{w}^2 (z\overline{w}-1)} - \frac{1}{\overline{z}^2 w^2 (\overline{z}w-1)} + \frac{1}{\overline{z}^2 \overline{w}^2 (\overline{z}\overline{w}-1)}  \right], \\
	& \frac{1}{n} \tr \left( \frac{R_{n,l}(z)}{2z^2} + \frac{\overline{R_{n,l}(z)}}{2\overline{z}^2} \right) \left( \frac{R_{n,l}(w)}{2\sqrt{-1}w^2} - \frac{\overline{R_{n,l}(w)}}{2\sqrt{-1}\overline{w}^2} \right)^\T \oindicator{\Omega_{n,l}} \\
	&\qquad \longrightarrow \frac{1}{4 \sqrt{-1}} \left[  \frac{1}{z^2 w^2 (zw-1)} - \frac{1}{z^2 \overline{w}^2 (z\overline{w}-1)} + \frac{1}{\overline{z}^2 w^2 (\overline{z}w-1)} - \frac{1}{\overline{z}^2 \overline{w}^2 (\overline{z}\overline{w}-1)}  \right],
\end{align*}
in probability as $n \rightarrow \infty$.  

Fix $r \geq 1$, pick $z_1, \ldots, z_r, w_1, \ldots, w_r \in D$, and let $\alpha_k, \beta_k \in \R$ for $1 \leq k \leq r$.  Then define the matrix
$$ Q := \sum_{k=1}^r \left[ \alpha_k \left( \frac{R_{n,l}(z_k)}{2z_k^2} + \frac{R_{n,l}(\bar{z_k})}{2\bar{z_k}^2} \right) + \beta_k \left( \frac{R_{n,l}(w_k)}{2 \sqrt{-1}w_k^2} - \frac{R_{n,l}(w_k)}{2 \sqrt{-1}\bar{w_k}^2} \right) \right] \oindicator{\Omega_{n,l}}. $$

Clearly $\tilde{\psi}$, $\tilde{\phi}$, and $Q$ are independent.  By applying Lemma \ref{lemma:clt} from Appendix \ref{sec:clt} to the term
$$ \frac{1}{\sqrt{n}} \tilde{\psi} Q \tilde{\phi} $$
and using the calculations above from Lemma \ref{lemma:variance}, it follows that the finite dimensional distributions of $S_n(z)$ converge to $S(z)$.

\section{Tightness} \label{sec:tight}

We now extend the finite dimensional convergence of $S_n(z)$ to weak convergence in the space of continuous functions by verifying that $\{S_n\}_{n \geq 1}$ is tight in the space of continuous functions on $D$ (see for example \cite{B-CPM}).  Since $S_n(z)$ is expressed as a linear combination of $Y_n(z)_{ij}$, $1 \leq i,j \leq l$, it suffices to show that $\{\|Y_n\|\}_{n \geq 1}$ is tight in the space of continuous functions on $D$. 

We remind the reader that
$$ Y_n(z) = \sqrt{n} \left[ \left( z - \frac{1}{\sqrt{n}} M_n \right)^{-1}_{ij} - \frac{1}{z} \delta_{ij} - \frac{m_{nij}}{z^2 \sqrt{n}} \right]_{i,j=1}^l. $$

Since $D$ is compact, there exists some constant $C(D) > 0$ such that
\begin{equation} \label{eq:4ez}
	2 + \epsilon \leq |z| \leq C(D) \text{ for all } z \in D. 
\end{equation}

By Proposition \ref{prop:contra},
$$ \Prob \left( \sup_{z \in D} \|Y_n(z) \| \oindicator{\Omega_n}^C > \eta \right) \leq \Prob \left( \sup_{z \in D} \| R_n(z) \| > \mathcal{K} \right) + \Prob(\oindicator{\Omega_n^C}) \longrightarrow 0 $$
as $n \rightarrow \infty$.  The problem then reduces to showing that $\{\|Y_n\| \oindicator{\Omega_n} \}_{n \geq 1}$ is tight in the space of continuous functions on $D$.  We will verify the Arzela-Ascoli criteria (\cite{B-CPM}) by checking that there exists constants $\alpha, \beta, C>0$ such that
$$ \E| \|Y_n(z)\| \oindicator{\Omega_n} - \|Y_n(w)\| \oindicator{\Omega_n}|^\alpha \leq C |z-w|^{1+\beta} $$
for all $z,w \in D$ and all $n \geq 1$. 

\begin{lemma}
There exists a constant $C>0$ (depending on the set $D$) such that
$$ \E| \|Y_n(z)\| \oindicator{\Omega_n} - \|Y_n(w)\| \oindicator{\Omega_n}|^2 \leq C |z-w|^2 $$
for all $z,w \in D$ and all $n \geq 1$.  
\end{lemma}

\begin{proof}
Define the $l \times l$ matrix 
$$ Q_n(z) := \left[ R_n(z)_{ij} - \frac{1}{z} \delta_{ij} \right]_{i,j=1}^l. $$
By the triangle inequality, 
\begin{align}
	|\|Y_n(z)\| \oindicator{\Omega_n} - \|Y_n(w)\| \oindicator{\Omega_n}| &\leq \| Y_n(z) - Y_n(w) \| \oindicator{\Omega_n} \nonumber \\
	& \leq \sqrt{n} \left\| Q_n(z) - Q_n(w) \right\|\oindicator{\Omega_n} + \left\| \frac{X}{z^2} - \frac{X}{w^2} \right\|. \label{eq:xz2}
\end{align}

We start by bounding the second term on the right-hand side of \eqref{eq:xz2}. For $z,w \in D$, 
$$ \left| \frac{\|X\|}{z^2} - \frac{\|X\|}{w^2} \right|^2 \leq \frac{\|X\|^2}{|z|^4|w|^4} |w^2 - z^2|^2 \leq \frac{\|X\|^2}{(2 + \epsilon)^8} 4C(D)^2 |w-z|^2. $$
We then apply the naive bound
\begin{equation} \label{eq:xet}
	\E\|X\|^2 \leq \E \tr (XX^\ast) = \sum_{i,j=1}^l \E|m_{nij}|^2 = l^2
\end{equation}
to obtain 
$$ \E \left| \frac{\|X\|}{z^2} - \frac{\|X\|}{w^2} \right|^2 \leq C |z-w|^2 $$
for some constant $C>0$ (which depends on the set $D$).

In order to deal with the first term on the right-hand side of \eqref{eq:xz2}, we decompose
\begin{align*}
	\sqrt{n}  \left[ R_n(z)_{ij} - \frac{1}{z} \delta_{ij} \right]_{i,j=1}^l &= \sqrt{n} \left( \left( z - \frac{1}{\sqrt{n}}X - \frac{1}{n} \psi R_{n,l}(z) \phi \right)^{-1} - \frac{1}{z}I_l \right) \\
	&= -\frac{1}{z} \left( z - \frac{1}{\sqrt{n}}X - \frac{1}{n} \psi R_{n,l}(z) \phi \right)^{-1} \left( X + \frac{1}{\sqrt{n}}\psi R_{n,l}(z) \phi \right) \\
	&= -\frac{1}{z} T_n(z) \left( X + \frac{1}{\sqrt{n}}\psi R_{n,l}(z) \phi \right)
\end{align*}
by the resolvent identity, where
$$ T_n(z) =  \left( z - \frac{1}{\sqrt{n}}X - \frac{1}{n} \psi R_{n,l}(z) \phi \right)^{-1}.$$
We note that all the relevant inverses above exist on the event $\Omega_n$.  Then
\begin{align*}
	\sqrt{n} \| Q_n(z) - Q_n(w) \| \leq \left\| \frac{X}{w}T_n(w) - \frac{X}{z}T_n(z) \right\| + \left\| \frac{T_n(w)}{w \sqrt{n}} \psi R_{n,l}(w) \phi - \frac{T_n(z)}{z\sqrt{n}} \psi R_{n,l} \phi \right \|
\end{align*}

On the event $\Omega_n$, we have that
\begin{align*}
	\left\| \frac{X}{w}T_n(w) - \frac{X}{z}T_n(z) \right\| &\leq \|X \| \left( \|T_n(w) \| \left| \frac{1}{w} - \frac{1}{z} \right| + \frac{1}{|z|} \|T_n(w) - T_n(z) \| \right)  \\
	& \leq \|X\| \left( \frac{\mathcal{K}}{(2+\epsilon)^2}|z-w| + \|T_n(w) - T_n(z) \| \right)
\end{align*}
and by the resolvent identity 
\begin{equation} \label{eq:tnw}
	\|T_n(w) - T_n(z) \| \leq \mathcal{K}^2 |z-w|. 
\end{equation}
Thus by \eqref{eq:xet}, we obtain that
$$ \E \left\| \frac{X}{w}T_n(w) - \frac{X}{z}T_n(z) \right\|^2 \oindicator{\Omega_n} \leq C |z-w|^2. $$

Similarly, we now bound
$$ \left\| \frac{T_n(w)}{w \sqrt{n}} \psi R_{n,l}(w) \phi - \frac{T_n(z)}{z\sqrt{n}} \psi R_{n,l} \phi \right \|. $$
Using the triangle inequality and \eqref{eq:tnw}, it suffices to bound 
\begin{equation} \label{eq:epr}
	\frac{1}{n} \E \| \psi R_{n,l}(z) \phi \|^2\oindicator{\Omega_n} 
\end{equation}
and
\begin{equation} \label{eq:epr2}
	\frac{1}{n} \E \left\| \psi \left[ R_{n,l}(z) - R_{n,l}(w) \right] \phi \right\|^2\oindicator{\Omega_n}.
\end{equation}

For \eqref{eq:epr}, we note that
\begin{align*}
	\frac{1}{n} \E \| \psi R_{n,l}(z) \phi \|^2\oindicator{\Omega_n} &\leq \frac{1}{n} \E \| \psi R_{n,l}(z) \phi \|^2\oindicator{\Omega_{n,l}} \\
	& \leq \frac{1}{n} \sum_{i,j=1}^l \E \left| (\psi R_{n,l}(z) \phi)_{ij} \right|^2 \leq l^2 \mathcal{K}^2
\end{align*}
by Lemma \ref{lemma:trace}.  

For \eqref{eq:epr2}, we apply the resolvent identity and obtain
\begin{align*}
	\frac{1}{n} \E \left\| \psi \left[ R_{n,l}(z) - R_{n,l}(w) \right] \phi \right\|^2\oindicator{\Omega_n} &\leq \frac{1}{n}|z-w|^2 \E \left\| \psi R_{n,l}(z) R_{n,l}(w) \phi \right\|^2\oindicator{\Omega_n} \\
	& \leq \frac{1}{n}|z-w|^2 \E \left\| \psi R_{n,l}(z) R_{n,l}(w) \phi \right\|^2\oindicator{\Omega_{n,l}}.
\end{align*}
Since $\psi R_{n,l}(z) R_{n,l}(w) \phi$ is an $l \times l$ matrix, we again apply Lemma \ref{lemma:trace} to each entry to bound the norm and obtain 
$$ \frac{1}{n} \E \left\| \psi R_{n,l}(z) R_{n,l}(w) \phi \right\|^2\oindicator{\Omega_{n,l}} \leq l^2 \mathcal{K}^4. $$
Therefore 
$$ \E\left\| \frac{T_n(w)}{w \sqrt{n}} \psi R_{n,l}(w) \phi - \frac{T_n(z)}{z\sqrt{n}} \psi R_{n,l} \phi \right \|^2 \oindicator{\Omega_n} \leq C |z-w|^2 $$
and the proof of the lemma is complete.
\end{proof}

\subsection*{Acknowledgment}
The author would like to thank an anonymous referee for the valuable suggestions and remarks.

\appendix

\section{Central Limit Theorem} \label{sec:clt}

Lemma \ref{lemma:clt} below is a Corollary of Theorem A.5 from \cite{ORS}.  The version presented here is for truncated random variables.  

\begin{lemma} \label{lemma:clt}
Let $\{ B^{s} : 1 \leq s \leq r\}$ be a family of $n \times n$ real random matrices.  Let $\{\mathbf{x}_n^{(s)}, \mathbf{y}_n^{(s)} : 1 \leq s \leq r \}$ be a collection of independent $n$-vectors with independent real standardized entries where $\mathbf{y}_n^{(s)} = (y_{nj}^{(s)})_{1 \leq j \leq n}$, $\mathbf{x}_n^{(s)} = (x_{nj}^{(s)})_{1 \leq j \leq n}$, and 
$$ \sup_{n,j} \left( \E|y_{nj}^{(s)}|^4 + \E|x_{nj}^{(s)}|^4 \right) \leq m_4 < \infty $$
for $1 \leq s \leq r$.  Further assume that
\begin{enumerate}[(i)]
\item for each $1 \leq s \leq r$, $B^s, \mathbf{x}_n^{(s)}$, annd $\mathbf{y}_n^{(s)}$ are independent,
\item there exists a sequence $\epsilon_n \rightarrow 0$ such that $\sup_{n,j,s} |x_{nj}^{(s)}| \leq \epsilon_n n^{1/4}$ and $\sup_{n,j,s} |y_{nj}^{(s)}| \leq \epsilon_n n^{1/4}$,
\item there exists a constant $a$ (not depending on $n$) such that $\max_{1 \leq s \leq r} \|B^s\| \leq a$,
\item for each $1 \leq s \leq r$, $\frac{1}{n} \tr \left( (B^s)^\ast B^s \right)$ converges in probability to a number $a_2(s)$.
\end{enumerate}
Then the random $r$-vector 
$$ \mathbf{Z}_n = \frac{1}{\sqrt{n}} \left( \left(\mathbf{x}_n^{(s)}\right)^\mathrm{T} B^s \mathbf{y}_n^{(s)} \right)_{1 \leq s \leq r} $$
converges in distribution as $n \rightarrow \infty$ to an $r$-vector $\mathbf{Z} = (z_s)_{1 \leq s \leq r}$ whose entries are independent normal random variables with mean zero and where the variance of $z_s$ is $a_2(s)$.  
\end{lemma}

\end{document}